\documentclass{article}
\usepackage{tikz}
\usetikzlibrary{positioning}
\usepackage{amssymb}
\usepackage{amsmath}
\usepackage{amsthm}
\usepackage{bm}
\usepackage{enumerate}
\usepackage{hyperref}

\theoremstyle{plain}
\newtheorem{thm}{Theorem}[section]

\theoremstyle{definition}
\newtheorem{defn}[thm]{Definition}

\newtheorem{corol}[thm]{Corollary}
\newtheorem{remar}[thm]{Remark}
\newtheorem{lemma}[thm]{Lemma}

\title{On the complexity of sets of uniqueness and extended uniqueness}
\author{João Paulos \\ Chalmers University of Technology}
\date{}

\begin{document}
\maketitle
\begin{abstract}
\noindent We locate the descriptive set theoretic complexity of the family of closed sets of uniqueness - $\mathcal{U}(G)\subseteq\mathcal{F}(G)$ - for locally compact connected Lie groups $G$ and of the family of closed sets of extended uniqueness - $\mathcal{U}_{0}(G)\subseteq\mathcal{F}(G)$ - for connected abelian Lie groups $G$, where $\mathcal{F}(G)$ denotes the set of closed subsets of $G$. More concretely, with respect to the Effros Borel space on $\mathcal{F}(G)$ we determine that under the aforementioned assumptions on $G$, $\mathcal{U}(G)$ and $\mathcal{U}_{0}(G)$ are $\Pi_{1}^{1}$-complete.
\end{abstract}

\section{Introduction}

\noindent The study of sets of uniqueness has a long and illustrious history. In his \emph{Habilitationsschrift}, Riemann suggested the problem of determining whether or not a representation of a function by a trigonometric series, whenever it exists, is unique. In modern terminology, this ammounts to ask whether or not the empty set is a set of uniqueness (for $\mathbb{T}$). Working on this problem, Cantor was led towards the creation of set theory. Lebesgue, Bernstein and Young strengthened Cantor results and eventually established that every countable subset of $\mathbb{T}$ is a set of uniqueness. The structure of sets of uniqueness of $\mathbb{T}$ was extensively studied during the first three decades of the last century, most notably by Russian and Polish schools. Menshov provided an example of a closed null set of multiplicity and Bari and Rajchman showed that there are non-empty perfect sets of uniqueness. It is fair to say that these results are not only counter-intuitive but also a faithful witness of the difficulties concerning a satisfactory characterization of sets of uniqueness, further enriching the depth of the subject. For a historical survey on the matter, which includes open problems reminiscent of the earlier days of this subject, the reader is referred to \cite{11}.  \\ During the 50's, the study of thin sets in harmonic analysis - including closed sets of uniqueness - gained a considerable new interest. Piatetski-Shapiro reformulated the language of the subject, bringing functional analytic methods into the study of sets of uniqueness of $\mathbb{T}$ and proving that there are closed sets of extended uniqueness which are not of uniqueness. Furthermore, the celebrated Salem-Zygmund theorem exposed a deep connection of this subject with Number Theory.\\ In \cite{12}, Herz extends the definition of set of uniqueness for abelian locally compact groups and in \cite{1}, Bozeijko extends it to general locally compact groups.  The study of sets of uniqueness is a remarkable example of successful application of techniques which appear naturally in different branches of research. In particular, the relation between the concepts of operator M-sets and sets of multiplicity is investigated in \cite{2}, providing an useful operator theoretic framework to study the properties of sets of uniqueness.\\ Curiously enough, even though the birth of set theory emerged from early developments in the subject, both areas remained relatively isolated from each other until almost a century later. In the 1980's, the incorporation of Descriptive Set Theory techniques revealed to be extremely fruitful in the study of sets of uniqueness of $\mathbb{T}$. The incorporation of descriptive set-theoretic tools, not only sharpened the understanding of the subject but also provided new insights from which earlier results appear from a more abstract context. The successful merging of analysis with descriptive set theoretic tools was led primarily by Debs-Saint Raymond, Kaufman, Kechris, Louveau, Woodin and Solovay. For a comprehensive introduction to this angle on sets of uniqueness, the reader is referred to \cite{3}. An important descriptive set-theoretic flavoured result which emerged from this approach was the classification of the complexity of the family of closed sets of uniqueness of $\mathbb{T}$ - $\mathcal{U}(\mathbb{T})$ - and of the family of closed sets of extended uniqueness of $\mathbb{T}$ - $\mathcal{U}_{0}(\mathbb{T})$ - both located at the level of $\Pi_{1}^{1}$-complete sets of $\mathcal{F}(\mathbb{T})$. In this paper, we prove these results in greater generality - i.e. we obtain such classifications for groups different than $\mathbb{T}$. More concretely, we prove that $\mathcal{U}(G)$ - for a locally compact connected Lie group $G$ - is $\Pi_{1}^{1}$-complete and that $\mathcal{U}_{0}(G)$ - for a connected abelia Lie group $G$ - is also $\Pi_{1}^{1}$-complete.  \\ \\ \noindent In Section 2, we set up our general framework : on one hand, we set up the modern operator-theoretic framework for sets of uniqueness in general locally compact groups and, on the other hand, we introduce the elementary descriptive set-theoretic context on which we will study the complexity of $\mathcal{U}(G)$ and $\mathcal{U}_{0}(G)$.\\ \\ In Section 3, we rely heavily on \cite{2} and \cite{6} in order to prove certain preservation properties of sets of uniqueness and sets of extended uniquess regarding finite direct products and inverse images. In particular, having in mind our end goal of classifying the complexity of $\mathcal{U}_{0}(G)$ for connected abelian Lie groups, we prove new results concerning these preservation properties for sets of extended uniqueness. \\ \\ In Section 4, we prove that $\mathcal{U}(G)$ and $\mathcal{U}_{0}(G)$ are coanalytic sets with respect to the Effros Borel space of $\mathcal{F}(G)$ under fairly general assumptions on $G$. Combining results from Section 3 with the more classic previously mentioned study of the complexity of $\mathcal{U}(\mathbb{T})$ and $\mathcal{U}_{0}(\mathbb{T})$, we establish the complexity of $\mathcal{U}(G)$ - for a locally compact connected Lie group $G$ - and of $\mathcal{U}_{0}(G)$ - for a connected abelian Lie group $G$ - both as $\Pi_{1}^{1}$-complete sets of $\mathcal{F}(G)$.

\section{Preliminaries}

\noindent In section 2.1 we set up a framework to study sets of uniqueness on arbitrary locally compact groups, thus extending the classical definition for $\mathbb{T}$. In section 2.2, we introduce terminology and concepts from classical Descriptive Set Theory that will be used in further sections concerning the study of the complexity of sets of uniqueness.

\subsection{The framework for sets of uniqueness}

We start by recalling the classical definition of \emph{set of uniqueness}. A subset $E\subseteq\mathbb{T}$ is said to be a set of uniqueness if the following holds : $$\sum_{n=-\infty}^{\infty}c_{n}e^{inx}=0, \ \text{for $x\notin E$} \Rightarrow \forall n\in\mathbb{Z} : c_{n}=0$$ where by equality with zero of the trigonometric series with coefficients $c_n\in\mathbb{C}$ we mean convergence of the partial sums $S_{N}=\sum_{n=-N}^{N}c_{n}e^{inx}$, as $N\rightarrow\infty$. \\ Recall that whenever such trigonometric series converges to zero on a set of positive measure, then $(c_n)\in c_{0}(\mathbb{Z})$ by Cantor-Lebesgue lemma. It is worth to mention that while every set of uniqueness has measure zero (cf. \cite{3}, Proposition 1 - I.3), the converse is false. Indeed, there are sets of measure zero which are not sets of uniqueness (cf. \cite{3}, Theorem 5 - III.4). The reader may thus find Definition \ref{unique} natural, as the classic case occurs when $G=\mathbb{T}$ (hence, $\hat{G}=\mathbb{Z}$). \\ If a subset $E\subseteq\mathbb{T}$ is not a set of uniqueness, then it is said to be a set of multiplicity. The collection of closed sets of uniqueness is denoted by $\mathcal{U}(\mathbb{T})$ and the collection of closed sets of multiplicity is denoted by $\mathcal{M}(\mathbb{T})$. Both collections will be considered in further sections as subsets of the hyperspace $\mathcal{K}(\mathbb{T})$ of closed subsets of the circle endowed with the Vietoris topology. \\ \\ Another important class of subsets of $\mathbb{T}$ are the so called \emph{sets of extended uniqueness}. A subset $E\subseteq\mathbb{T}$ is said to be a set of extended uniqueness if the following holds : $$\sum_{n=-\infty}^{\infty}\hat{\mu}(n)e^{inx}=0, \ \text{for $x\notin E$} \Rightarrow \forall n\in\mathbb{Z} : \hat{\mu}(n)=0$$ where $\mu$ is a Borel measure on $\mathbb{T}$ and $\hat{\mu}(n)$ denote its Fourier-Stieltjes coefficients. A set which is not a set of extended uniqueness is said to be a set of restricted multiplicity. The family of closed sets of extended uniqueness is denoted by $\mathcal{U}_{0}(\mathbb{T})$ and we set $\mathcal{M}_{0}(\mathbb{T})=\mathcal{K}(\mathbb{T})\setminus\mathcal{U}_{0}(\mathbb{T})$.\\ \\ It is clear that $\mathcal{U}(\mathbb{T})\subseteq\mathcal{U}_{0}(\mathbb{T})$ and in fact, the inclusion is strict (cf. \cite{3}, Corollary 17 - VII.3). \\ \\ We now extend these notions of uniqueness for locally compact groups $G$. \\ \\ Henceforth, $M(G)$ denotes the Banach algebra of complex Borel measures on $G$ with respect to convolution of measures. We will denote by $L^{p}(G)$ the Lebesgue space with respect to the Haar measure on $G$ and we recall that $L^{1}(G)$ is canonically identified as a closed ideal of $M(G)$. As usual, $\lambda : G\rightarrow \mathcal{L}(L^{2}(G))$ with $\lambda_{s}:=\lambda(s)$ denotes the left regular representation of $G$. We also use $\lambda$ to denote the representation of $M(G)$ on $L^{2}(G)$ given by : $$\lambda(\mu)(\xi)(g)=\int_{G}\xi(s^{-1}g)d\mu(s)=\int_{G}(\lambda_{s}\xi)(g)d\mu(s)$$ for $\mu\in M(G)$ and $\xi\in L^{2}(G)$. \\ The reduced group C$^{\ast}$-algebra of $G$ is the following C$^{\ast}$-subalgebra : $$C_{r}^{\ast}(G)=\overline{\lbrace\lambda(f):f\in L^{1}(G)\rbrace}^{||.||}$$ The group von Neumann algebra of $G$ is denoted by $VN(G)$ and defined as $VN(G)=\overline{C_{r}^{\ast}(G)}^{w^{\ast}}$. The group C$^{\ast}$-algebra of $G$, denoted by $C^{\ast}(G)$, is the completion of $L^{1}(G)$ under the norm $||f||_{C^{\ast}}:=\sup_{\pi}\lbrace ||\pi(f)||\rbrace$, where the supremum is taken over all unitary equivalence classes of irreducible representations $\pi$ of  $G$.\\ The Fourier-Stieltjes algebra of $G$ will be denoted by $B(G)$. As a set it consists of all functions $u:G\rightarrow\mathbb{C}$ of the form : $$u(s)=(\pi(s)\xi,\eta)$$ where $\pi:G\rightarrow\mathcal{L}(H)$ is a continuous unitary representation acting on some Hilbert space $H$ and $\xi,\eta\in H$. \\ The Fourier algebra of $G$ will be denoted by $A(G)$ and as a set it consists of all functions $u:G\rightarrow\mathbb{C}$ of the form : $$u(s)=(\lambda_{s}\xi,\eta)$$ for $\xi,\eta\in L^{2}(G)$. Those are algebras with respect to pointwise multiplication. One has that $A(G)$ is a closed ideal of $B(G)$ and that $A(G)$ can be identified with the predual of $VN(G)$ via the pairing :  $$\langle u,T\rangle =(T\xi, \eta)$$ where $u\in A(G)$ has the form $u(s)=(\lambda_{s}\xi,\eta)$. \\ If $T\in VN(G)$ and $u\in A(G)$, the operator $u.T\in VN(G)$ is given by the relation $\langle u.T,v\rangle =\langle T,uv\rangle$, for $v\in A(G)$. \\ \\ For further details and properties about these algebras, we refer the interested reader to \cite{2} and \cite{4}. \\ \\  We point out that if $G$ is abelian, then $C_{r}^{\ast}(G)$ coincides with $C_{0}(\hat{G})$ - where $\hat{G}$ is the dual group of $G$ - and $VN(G)$ with $L^{\infty}(\hat{G})$. Furthermore, in this case the Fourier transform gives an isometric *-isomorphism between $L^{1}(G)$ and $A(\hat{G})$. In particular, when $G$ is a locally compact abelian group we get that $L^{1}(\hat{G})\approx A(G)$. For instance, with $G=\mathbb{T}^{d}$ we recover the duality : $$A(\mathbb{T}^{d})^{\ast}\approx \ell^{1}(\mathbb{Z}^{d})^{\ast}\approx \ell^{\infty}(\mathbb{Z}^{d})\approx VN(\mathbb{T}^{d})$$

\begin{defn}\label{suporte2} Let $T\in VN(G)$ and $g\in G$. We define $supp^{\ast}(T)$ as the set of elements in $G$ which satisfy one of the following equivalent conditions (cf. Proposition 4.4, in \cite{4}) : \begin{enumerate}[(i)]\item{For every open neighborhood $\mathcal{V}$ of $g$, there is some element $u\in A(G)$ with supp($u$)$\subseteq\mathcal{V}$ and such that $\langle u,S\rangle\neq 0$.}\item{$u.T\neq 0$, whenever $u\in A(G)$ and $u(g)\neq 0$.}\end{enumerate}\end{defn}

\noindent For the sake of completeness, we list some properties of $supp^{\ast}(T)$, for $T\in VN(G)$ that will be used in further sections : \begin{lemma}\label{propsup} Let $T\in VN(G)$, then $supp^{\ast}(T)$ has the following properties : \begin{enumerate}[(i)]\item{If $T\neq 0$, then $supp^{\ast}(T)$ is a non-empty closed set.}\item{$supp^{\ast}(T)$ is the smallest closed set $E\subseteq G$ such that for all $u\in A(G)\cap C_{c}(G)$ with $\text{supp}(u)\cap E=\emptyset$ then $\langle u,T\rangle = 0$.}\item{If $T=\lambda(\mu)$ for $\mu\in M(G)$, then $supp^{\ast}(T)=\text{supp}(\mu)$.}\end{enumerate}\end{lemma}\begin{proof} The reader can find a proof in \cite{4} (respectively, Prop. 4.6, Prop. 4.8 and Remark 4.7).\end{proof}

\noindent For a closed subset $E\subseteq G$ we will use the following notation : $$J(E)^{\perp}=\lbrace T\in VN(G) : supp^{\ast}(T)\subseteq E\rbrace$$ For a motivation of such notation, we refer the interested reader to \cite{2}. \\ 

\noindent We adopt the definition sets of uniqueness which was introduced in \cite{1} and further generalized in \cite{2} : 

\begin{defn}\label{unique} Let $E\subseteq G$ be a closed subset. Then, $E$ is a set of uniqueness if $S\in C_{r}^{\ast}(G)$ and $supp^{\ast}(S)\subseteq E$, imply that $S=0$. Any closed set which is not a set of uniqueness, will be called a set of multiplicity (or a $M$-set). We denote the family of closed sets of uniqueness of $G$ by $\mathcal{U}(G)$ and the family of closed sets of multiplicity of $G$ by $\mathcal{M}(G)$. \end{defn}

\noindent For a closed subset $E\subseteq G$, we let $M(E)\subseteq M(G)$ be the subset of measures with support contained in $E$ and we adopt the definition of sets of extended uniqueness given in \cite{2} :

\begin{defn}\label{extendeduniq} Let $E\subseteq G$ be a closed subspace. Then, $E$ is a set of extended uniqueness if $\lambda(M(E))\cap C_{r}^{\ast}(G)=\lbrace 0\rbrace$. Any closed set which is not a set of extended uniqueness, will be called a set of restricted multiplicity (or a $M_0$-set). We denote the family of closed sets of extended uniqueness of $G$ by $\mathcal{U}_{0}(G)$ and the family of closed sets of restricted multiplicity of $G$ by $\mathcal{M}_{0}(G)$.\end{defn}

\begin{remar}\label{remarkk} Indeed, Definitions \ref{unique} and \ref{extendeduniq} generalize the classical notions : \begin{enumerate}[(i)]\item{Let $PM(\mathbb{T})$ be the space of pseudomeasures and $PF(\mathbb{T})$ be the space of pseudofunctions. Moreover, for a closed subset $E\subseteq\mathbb{T}$, let $PM(E)$ denote the subset of pseudomeasures supported on $E$. It is known that $PM(E)=J(E)^{\perp}$ (see \cite{14}).}\item{If $G$ is abelian or compact then the class of measures $\mu\in M(G)$ which satisfy $\lambda(\mu)\in C_{r}^{\ast}(G)$ coincide with the class of Rajchman measures, i.e. measures whose Fourier-Stieltjes coefficients vanish at infinity (see \cite{13}).}\item{One can reformulate the classical definition for elements in $\mathcal{M}(\mathbb{T})$ and $\mathcal{M}_{0}(\mathbb{T})$ as follows (cf. \cite{3}, Theorem 1 - II.4 and Propostion 6 - II.5): $$E\in\mathcal{M}(\mathbb{T})\Leftrightarrow PM(E)\cap PF(\mathbb{T})\neq\lbrace 0\rbrace$$ $$E\in\mathcal{M}_{0}(\mathbb{T})\Leftrightarrow M(E)\cap PF(\mathbb{T})\neq\lbrace 0\rbrace$$ Recalling that if $G$ is abelian one has that $C_{r}^{\ast}(G)\approx C_0(\hat{G})$, the conclusion follows from (i) and (ii).}\end{enumerate}\end{remar}

\subsection{Borel hierarchy and analytic sets}

\begin{defn} A Polish space is a completely metrizable separable topological space. A measurable space $(X,\Sigma)$ is a standard Borel space if there is some Polish topology $\mathcal{T}$ on $X$ such that its Borel algebra - i.e. the smallest $\sigma$-algebra containing $\mathcal{T}$ - coincides with $\Sigma$. In other words, if $\mathcal{B}(\mathcal{T})=\Sigma$.\end{defn}

\noindent The study of regularity and definability properties on Polish spaces plays a central role in classical Descriptive Set Theory. For completeness, we mention some examples of Polish spaces which will appear in further sections : \\ \\ $\bullet$ A topological group\footnote{We assume that topological groups are Hausdorff.} is called a Polish group if it is Polish as a topological space. Any locally compact group which is second countable is Polish. In particular, every connected (hence second countable) locally compact Lie group and every connected abelian Lie group (being isomorphic to $\mathbb{R}^{n}\times\mathbb{T}^{m}$ for some positive integers $n,m$) is Polish. \\ \\ $\bullet$ If $X$ is a topological space, let $\mathcal{K}(X)$ be the collection of its compact subsets. The Vietoris topology on $\mathcal{K}(X)$ is the topology generated by subsets of the following form : $$\lbrace K\in\mathcal{K}(X):K\subseteq\mathcal{U}\rbrace \ \ \text{and} \ \ \lbrace K\in\mathcal{K}(X):K\cap\mathcal{U}\neq\emptyset\rbrace$$ for an open set $\mathcal{U}\subseteq X$. The following is a well known result : \begin{thm} Let $X$ be a Polish space. Then, $\mathcal{K}(X)$ endowed with the Vietoris topology is a Polish space.\end{thm}\begin{proof} The reader can find a proof in \cite{9} (Theorem 4.25).\end{proof} \noindent $\bullet$ If $X$ is a topological space, let $\mathcal{F}(X)$ be the collection of its closed subsets. The collection $\mathcal{F}(X)$ endowed with the $\sigma$-algebra generated by sets of the following form : $$\lbrace F\in\mathcal{F}(X):F\cap\mathcal{U}\neq\emptyset\rbrace,  \ \text{for open sets $\mathcal{U}\subseteq X$}$$ is called the Effros space on $\mathcal{F}(X)$, which we will denote by $\mathcal{E}(X)$. The relevance of this example emerges from the next two well known results : \begin{thm} Let $X$ be a Polish space. Then, $\mathcal{F}(X)$ is a standard Borel space.\end{thm}\begin{proof} The reader can find a proof in \cite{9} (Theorem 12.6).\end{proof} \begin{thm} Let $X$ be a compact Polish space. Then, $\mathcal{E}(X)$ on $\mathcal{F}(X)=\mathcal{K}(X)$ coincides with the Borel algebra generated by the Vietoris topology.\end{thm}\begin{proof} The reader can find a proof in \cite{9} (Example 12.7).\end{proof} \noindent If furthermore $X$ is locally compact, $\mathcal{E}(X)$ is induced by the Fell topology (which is a compact metrizable topology that coincides with the Vietoris topology whenever $X$ is compact) on $\mathcal{F}(X)$ which has as a basis sets of the form : $$\lbrace F\in\mathcal{F}(X): F\cap K=\emptyset, F\cap\mathcal{U}_{1}\neq\emptyset,...,F\cap\mathcal{U}_{n}\neq\emptyset\rbrace$$ with $K$ varying over $\mathcal{K}(X)$ and each $\mathcal{U}_{i}\subseteq X$ an open set.\\ \\
\noindent For a Polish space $X$, its Borel algebra $\mathcal{B}(X)$ is stratified in the so called Borel hierarchy. Given a family of subsets $\mathcal{F}\subseteq\mathcal{P}(X)$ we adopt the usual notations : $$\mathcal{F}_{\sigma}=\lbrace \bigcup_{n\in\omega}A_{n}:A_{n}\in\mathcal{F}\rbrace$$ $$\neg\mathcal{F}=\lbrace \neg A : A\in\mathcal{F}\rbrace$$ We define the first level of the Borel hierarchy as follows : $$\Sigma_{1}^{0}(X)=\lbrace\mathcal{U}\subseteq X : \mathcal{U}\ \text{is open}\rbrace \ \text{and} \ \Pi_{1}^{0}(X)=\neg\Sigma_{1}^{0}(X)$$ We proceed by transfinite recursion, for $1<\alpha<\omega_{1}$ as follows : $$\Sigma_{\alpha}^{0}(X)=(\bigcup_{\beta<\alpha}\Pi_{\beta}^{0}(X))_{\sigma} \ \text{and} \ \Pi_{\alpha}^{0}(X)=\neg\Sigma_{\alpha}^{0}(X)$$ It is straightforward to verify that if $\beta<\alpha$, then $\Sigma_{\beta}^{0}(X)\subseteq\Sigma_{\alpha}^{0}(X)$ and that any Borel set of $X$ lies in some level of the Borel hierarchy : $$\mathcal{B}(X)=\bigcup_{\alpha <\omega_{1}}\Sigma_{\alpha}^{0}(X)=\bigcup_{\alpha <\omega_{1}}\Pi_{\alpha}^{0}(X)$$ The class of analytic sets of $X$, denoted by $\Sigma_{1}^{1}(X)$, is defined as follows : $$\Sigma_{1}^{1}(X)=\lbrace A\subset X : A=\pi(B), B\in\mathcal{B}(X\times Y) \ \text{for some Polish space $Y$}\rbrace$$ where $\pi$ denotes the projection map. The class of coanalytic sets of $X$, denoted by $\Pi_{1}^{1}(X)$, is defined as $$\Pi_{1}^{1}(X)=\neg\Sigma_{1}^{1}(X)$$ An equivalent and often convenient definition of an analytic set is as follows : \\ \\ Let $X$ be Polish and $A\subseteq X$. Then, $A$ is analytic if and only if there is a Polish space $P$ and a continuous map $f:P\rightarrow X$ such that $f(P)=A$. \\ \\ It turns out that the class of Borel sets of a Polish space coincides with the class of those sets which are simultaneously analytic and coanalytic : \begin{thm} Let $X$ be a Polish space. Then, $\mathcal{B}(X)=\Sigma_{1}^{1}(X)\cap\Pi_{1}^{1}(X)$.\end{thm}\begin{proof} The reader can find a proof in \cite{9} (Theorem 14.11).\end{proof} \noindent Furthermore, if $X$ is uncountable then $\Sigma_{1}^{1}(X)$ and $\Pi_{1}^{1}(X)$ extend $\mathcal{B}(X)$. This was famously proved by Suslin after detecting a mistake in Lebesgue's work.\footnote{Lebesgue claimed in \cite{15} that the projection of a Borel set of the plane onto the real line is again a Borel set.} \begin{thm}\label{su} Let $X$ be an uncountable Polish space. Then, $\mathcal{B}(X)\subsetneq\Pi_{1}^{1}(X)$.\end{thm}
\begin{proof} The reader can find a proof in \cite{9} (Theorem 14.12).\end{proof}

\noindent One can argue through different lenses that sets which are (co)analytic and not Borel are of \emph{high complexity}, thus their detection being a topic of interest. A commonly used technique to identify these sets involves the concept of $\Gamma$-hardness. In what follows, we consider $\Gamma(X)$ to be either $\Sigma_{1}^{1}(X)$ or $\Pi_{1}^{1}(X)$. \begin{defn} Let $X$ be a Polish space and $A\subseteq X$. Then, $A$ is said to be $\Gamma(X)$-hard if for every zero-dimensional Polish space $Y$ and $B\in\Gamma(Y)$, there is a continuous map $f:Y\rightarrow X$ such that $B=f^{-1}(A)$. If, furthermore $A\in\Gamma(X)$, then $A$ is said to be $\Gamma(X)$-complete.\end{defn}

\noindent We note that if $A\subseteq X$ is a $\Gamma(X)$-complete set, then $A\in\Gamma(X)\setminus\mathcal{B}(X)$. Indeed, by Theorem \ref{su} there exists a set $B\in\Gamma(\omega^{\omega})\setminus\mathcal{B}(\omega^{\omega})$. If $A$ is $\Gamma(X)$-hard, then there is a continuous map $f:\omega^{\omega}\rightarrow X$ such that $B=f^{-1}(A)$. Since continuous preimages preserve the property of being Borel, we conclude that $A\notin\mathcal{B}(X)$. \\ \\ 
\noindent We  finish the section with the following example of a $\Pi_{1}^{1}(\mathcal{K}([0,1]))$-complete set, which shall be used further :

\begin{thm}\label{redu} Let $\mathbb{Q}'=\mathbb{Q}\cap [0,1]$. Then, the following set is $\Pi_{1}^{1}(\mathcal{K}([0,1]))$-complete (wrt Vietoris topology) : $$K(\mathbb{Q}')=\lbrace K\in\mathcal{K}([0,1]):K\subseteq\mathbb{Q}'\rbrace$$\end{thm}
\begin{proof} The reader can find a proof in \cite{3} (Theorem IV.2.1).\end{proof}

\section{Sets of operator multiplicity and sets of multiplicity}

In section 3.1 we introduce two notions of sets of operator multiplicity and relate these with sets of multiplicity as introduced in section 2.1. In section 3.2, we prove some preservation properties concerning direct products and inverse images.

\subsection{Definitions}

\noindent The concepts of \emph{operator $M$-sets} and \emph{operator $M_0$-sets} were introduced in \cite{5} and are of special importance as they allow the use of operator theoretic methods in the study of questions emerging from Harmonic Analysis. We refer the reader to \cite{2},\cite{6},\cite{7} and \cite{8} for further details. In this section, we state the relevant definitions and results needed in further sections, mainly following \cite{2}. \\ \\
\noindent A measure space $(X,\mu)$ will be called a standard measure space if $\mu$ is a Radon measure with respect to some completely metrizable separable and locally compact topology on $X$. For standard measure spaces $(X,\mu)$ and $(Y,\nu)$, a subset $R\subseteq X\times Y$ is said to be a rectangle if it is of the form $R=A\times B$ for $A$ and $B$ measurable. Endowing the product $X\times Y$ with the product measure : \begin{enumerate}[(i)]\item{} $E\subseteq X\times Y$ is said to be \emph{marginally null} if $E\subseteq (X_{0}\times Y)\cup (X\times Y_{0})$ for $\mu(X_0)=0$ and $\nu(Y_0)=0$. \item{} $E,F\subseteq X\times Y$ are said to be \emph{marginally equivalent} if $E\Delta F$ is marginally null. We denote this by $E\sim F$. \item{} A subset $E\subseteq X\times Y$ is said to be $\omega$-open if it is marginally equivalent to a countable union of rectangles. The complement of a $\omega$-open set will be called $\omega$-closed.\end{enumerate}
\noindent Given Hilbert spaces $\mathcal{H}_{1}$ and $\mathcal{H}_{2}$, we denote as usual the space of compact operators in $\mathcal{L}(\mathcal{H}_{1},\mathcal{H}_2)$, by $\mathcal{K}(\mathcal{H}_{1},\mathcal{H}_{2})$. Henceforth, $\mathcal{H}_{1}=L^{2}(X,\mu)$ and $\mathcal{H}_{2}=L^{2}(Y,\nu)$. The space $\mathcal{C}_{1}(\mathcal{H}_2,\mathcal{H}_1)$ of nuclear operators is identified with the Banach space dual of $\mathcal{K}(\mathcal{H}_1,\mathcal{H}_2)$ via $\langle T,S\rangle = \text{tr}(TS)$. Moreover, one can identify $\mathcal{C}_{1}(\mathcal{H}_2,\mathcal{H}_1)$ with the space $\Gamma(X,Y)$ of all (marginal equivalence classes of) fuctions $h:X\times Y\rightarrow\mathbb{C}$ which admit a representation : $$h(x,y)=\sum_{i=1}^{\infty}f_{i}(x)g_{i}(y)$$ with $f_i\in\mathcal{H}_1$ and $g_i\in\mathcal{H}_2$ such that $\sum_{i=1}^{\infty}||f_i||_{2}^{2}<\infty$ and $\sum_{i=1}^{\infty}||g_i||_{2}^{2}<\infty$. The duality between $\mathcal{L}(\mathcal{H}_1,\mathcal{H}_2)$ and $\Gamma(X,Y)$ is given by : $$\langle T,f\otimes g\rangle =(Tf,\overline{g}), \ \text{for $T\in\mathcal{L}(\mathcal{H}_1,\mathcal{H}_2)$ and $f\in L^{2}(X,\mu),g\in L^{2}(Y,\nu)$}$$ If $f\in L^{\infty}(X,\mu)$, let $M_{f}\in\mathcal{B}(\mathcal{H}_{1})$ be the operator of multiplication by $f$. The collection $\lbrace M_{f}\rbrace_{f\in L^{\infty}(X,\mu)}$ is a maximal abelian selfadjoint algebra (masa). If $A\subseteq X$ is measurable, we write $P(A)=M_{\chi_{A}}$ where $\chi_{A}$ is the characteristic map of $A$. A subspace $W\subseteq\mathcal{B}(\mathcal{H}_{1},\mathcal{H}_{2})$ is called a masa-bimodule if $M_{\Psi}TM_{\varphi}\in W$ for all $\Psi\in L^{\infty}(X,\mu)$, $T\in W$ and $\varphi\in L^{\infty}(Y,\nu)$. \\ \\ Now let $E\subseteq X\times Y$ be a $\omega$-closed set and $T\in\mathcal{B}(\mathcal{H}_{1},\mathcal{H}_{2})$. We say that $E$ supports $T$ (or that $T$ is supported on $E$) if $P(B)MP(A)=0$ whenever $A\times B\cap E\sim\emptyset$. For a subset $\mathcal{M}\subseteq\mathcal{B}(\mathcal{H}_{1},\mathcal{H}_{2})$, there exists a smallest (up to marginal equivalence) $\omega$-closed set which support every operator $T\in\mathcal{M}$, that we denote by $supp(\mathcal{M})$. On the other hand, it is known that for every $\omega$-closed set $E$ there exists the smallest and the largest weak$^{\ast}$ closed masa-bimodule - respectively $\frak{M}_{min}(E)$ and $\frak{M}_{max}(E)$ - with support $E$. \\ \\ \noindent We are now in position to introduce operator $M$-sets : \begin{defn} Let $(X,\mu)$ and $(Y,\nu)$ be standard measure spaces and $\kappa\subseteq X\times Y$ be a $\omega$-closed set. Then, $\kappa$ is an operator M-set if : $$\mathcal{K}(\mathcal{H}_1,\mathcal{H}_2)\cap\frak{M}_{max}(\kappa)\neq\lbrace 0\rbrace$$ Otherwise, $\kappa$ is said to be an operator U-set.\end{defn}

\noindent In order to define operator $M_0$-sets, we need some additional terminology. We follow \cite{16} and consider $\sigma$ to be a complex measure of finite total variation defined on the product $\sigma$-algebra of $X\times Y$ and let $|\sigma|$ denote the variation of $\sigma$ and $|\sigma|_{X},|\sigma|_{Y}$ be the marginal measures of $|\sigma|$. Such measure $\sigma$ is said to be an Arveson measure if there is a constant $c>0$ such that the following hold : $$|\sigma|_{X}\leq c\mu \ \ \text{and} \ \ |\sigma|_{Y}\leq c\nu$$ The set of all Arveson measures on $X\times Y$ is denoted by $\mathbb{A}(X,Y)$ and for some $\sigma\in\mathbb{A}(X,Y)$, we denote the smallest constant satisfying its defining inequalities by $||\sigma||_{\mathbb{A}}$. For a $\sigma$-closed subset $F\subseteq X\times Y$, we denote by $\mathbb{A}(F)$ the set of all Arveson measures $\sigma$ in $X\times Y$ such that $supp(\sigma)\subseteq F$. \\ \\ \noindent An Arveson measure $\sigma\in\mathbb{A}(Y,X)$ defines an operator $T_\sigma :\mathcal{H}_1\rightarrow\mathcal{H}_2$ which will be called pseudointegral. These operators were introduced in \cite{16}. Indeed, for $\sigma\in\mathbb{A}(Y,X)$ one can consider the sesquilinear form $\phi:\mathcal{H}_{1}\times\mathcal{H}_2\rightarrow\mathbb{C}$ given by : $$\phi(f,g)=\int_{Y\times X}f(x)\overline{g(y)}d\sigma(y,x)$$ By Riesz Representation Theorem, it follows that there is an unique operator $T_{\sigma}:\mathcal{H}_1\rightarrow\mathcal{H}_2$ such that $(T_{\sigma}f,g)=\phi(f,g)$. \\ \\ For a given $\omega$-closed subset $\kappa\subseteq X\times Y$ we let $\hat{\kappa}=\lbrace (y,x):(x,y)\in\kappa\rbrace$. We can now state the following :

\begin{thm} Let $\sigma\in\mathbb{A}(Y,X)$. There exists an unique $T_{\sigma}:\mathcal{H}_{1}\rightarrow\mathcal{H}_2$ such that : $$(T_{\sigma}f,g)=\int_{Y\times X}f(x)\overline{g(y)}d\sigma(y,x), \ \text{for $f\in\mathcal{H}_{1},g\in\mathcal{H}_2$}$$ Moreover, $||T_{\sigma}||\leq ||\sigma||_{\mathbb{A}}$ and for a given $\omega$-closed subset $\kappa\subseteq X\times Y$ the operator $T_{\sigma}$ is supported on $\kappa$ if and only if $supp(\sigma)\subseteq\hat{\kappa}$.\label{arvesonop}\end{thm}
\begin{proof} The reader can find a proof in \cite{2} (Theorem 3.2).\end{proof}

\noindent We can finally define operator $M_0$-sets :

\begin{defn} Let $(X,\mu)$ and $(Y,\nu)$ be standard measure spaces and $\kappa\subseteq X\times Y$ be a $\omega$-closed set. Then, $\kappa$ is an operator $M_{0}$-set if : $$\text{There is a non-zero measure $\sigma\in\mathbb{A}(\hat{\kappa})$ such that $T_{\sigma}\in\mathcal{K}(\mathcal{H}_1,\mathcal{H}_2)$}$$ Otherwise, $\kappa$ is said to be an operator $U_{0}$-set.\end{defn}

\noindent We finish this section with the bridge between sets of multiplicity and their operator theoretic counterpart. Let $G$ be a group and $E\subseteq G$. We define : $$E^{\ast}=\lbrace (s,t)\in G\times G : ts^{-1}\in E\rbrace$$ We note that if $G$ is second countable and $E$ is closed, then $E^{\ast}$ is $\omega$-closed. We have the following central results :

\begin{thm}\label{teorema1} Let $G$ be a locally compact second countable group and $E\subseteq G$ be a closed subset. Then, the following are equivalent :\begin{enumerate}[(i)]\item{$E$ is a $M$-set.}\item{$E^{\ast}$ is an operator M-set.}\end{enumerate}\end{thm}
\begin{proof} The reader can find a proof in \cite{2} (Theorem 4.9).\end{proof}

\begin{thm}\label{operatorM0versusM0} Let $G$ be a locally compact second countable group and $E\subseteq G$ be a closed subset. Then, the following are equivalent : \begin{enumerate}[(i)]\item{$E$ is a $M_{0}$-set.}\item{$E^{\ast}$ is an operator $M_0$-set.}\end{enumerate}\end{thm}
\begin{proof} The reader can find a proof in \cite{2} (Theorem 4.12).\end{proof}

\subsection{Preservation properties}

\noindent In this section we prove some preservation properties involving direct products and inverse images. We divide the section into two subsections. In section 3.2.1 we essentially follow \cite{2} and deal with $M$-sets. In section 3.2.2, we establish some new results concerning the aforementioned preservation properties for $M_0$-sets.

\subsubsection{Operator $M$-sets and $M$-sets}

Let $(X_{i},\mu_{i})$ and $(Y_{i},\nu_{i})$ be standard measure spaces. We define the following map : $$\rho : (X_{1}\times Y_{1})\times (X_{2}\times Y_{2})\rightarrow (X_{1}\times X_{2})\times (Y_{1}\times Y_{2})$$ $$((x_{1},y_{1}),(x_2,y_2))\mapsto ((x_1,x_2),(y_1,y_2))$$ We note that the following useful identity holds : $$\rho(E_{1}^{\ast}\times E_{2}^{\ast})=(E_{1}\times E_{2})^{\ast}$$ for $E_{i}\subseteq G_{i}$ where $G_1$ and $G_2$ are groups. Since we already know how to relate operator $M$-sets with $M$-sets (section 3.1), the key element to prove preservation of $M$-sets under finite direct products is the following result : 
\begin{thm}\label{teorema2} Let $(X_{i},\mu_{i})$ and $(Y_i,\nu_{i})$ be standard measure spaces and $E_{i}\subseteq X_{i}\times Y_{i}$ be $\omega$-closed sets. The set $\rho(E_{1}\times E_{2})$ is an operator M-set if and only if both $E_1$ and $E_2$ are operator M-sets.\end{thm}
\begin{proof} The reader can find a proof in \cite{2} (Theorem 5.11).\end{proof}

\begin{corol}\label{prod}Let $G_1$ and $G_2$ be locally compact second countable groups and let $E_1$ and $E_2$ be closed sets. Then, the following holds : \begin{enumerate}[(i)]\item{If $E_{1}$ or $E_{2}$ are sets of uniqueness, then $E_{1}\times E_{2}$ is a set of uniqueness}\item{If $E_1$ and $E_2$ are sets of multiplicity, then $E_{1}\times E_{2}$ is a set of multiplicity}\end{enumerate}\end{corol}\begin{proof} We only prove (i), since the argument for (ii) is entirely analogous (see \cite{2}, Corollary 5.12). Suppose that $E_1$ is a set of uniqueness. By Theorem \ref{teorema1}, $E_{1}^{\ast}$ is an operator U-set. Since $\rho(E_{1}^{\ast}\times E_{2}^{\ast})=(E_{1}\times E_{2})^{\ast}$, it follows from Theorem \ref{teorema2} that $(E_{1}\times E_{2})^{\ast}$ is an operator U-set and thus, by Theorem \ref{teorema1}, we conclude that $E_{1}\times E_{2}$ is a set of uniqueness.\end{proof}

\noindent Concerning inverse images, the following two results will be of crucial importance in further sections :

\begin{thm}\label{inv} Let $(X,\mu),(Y,\nu),(X_1,\mu_1),(Y_1,\nu_1)$ be standard Borel spaces and suppose that $\varphi:X\rightarrow X_1$ and $\Psi:Y\rightarrow Y_1$ are measurable maps. Let $E\subseteq X_{1}\times Y_1$ and $F=\lbrace (x,y)\in X\times Y : (\varphi(x),\Psi(y))\in E\rbrace$. If $\varphi_{\ast}\mu$ and $\Psi_{\ast}\nu$ are equivalent, respectively, to $\mu_{1}$ and $\nu_{1}$, then $E$ is an operator M-set if and only if $F$ is an operator M-set.\end{thm}
\begin{proof} The result follows from Theorem 5.5 and Remark 5.6 in \cite{2}.\end{proof}

\begin{corol}\label{isoMsets} Let $G$ and $H$ be locally compact second countable groups with Haar measures $m_{G}$ and $m_{H}$ respectively and $E\subseteq H$ be a closed set. If $\varphi :G\rightarrow H$ is a continuous isomorphism, then $E$ is a $M$-set if and only if $\varphi^{-1}(E)$ is a $M$-set.\end{corol}
\begin{proof} Since $\varphi$ is a homomorphism, $\varphi^{-1}(E)^{\ast}=(\varphi^{-1}\times\varphi^{-1})(E^{\ast})$ and since it is an isomorphism, $\varphi_{\ast}m_{G}$ and $m_{H}$ are equivalent. Thus, the result follows from Theorem \ref{teorema1} and Theorem \ref{inv}. \end{proof}

\subsubsection{Operator $M_0$-sets and $M_0$-sets}

\noindent In this section we study preservation properties of $M_0$-sets under inverse images and finite direct products, towards a classification of the complexity of $\mathcal{U}_{0}(G)$, for $G$ connected abelian Lie group. Every such group is isomorphic, for some non-negative integers $m,n$, to $\mathbb{R}^{n}\times\mathbb{T}^{m}$. Consequently, we are mostly concerned with inverse images of the quotient map $q:\mathbb{R}^{n}\times\mathbb{T}^{m}\rightarrow\mathbb{T}^{n+m}$ and finite products of $\mathbb{T}$. Nevertheless, we establish some preservation results with greater generality. \\ \\ We recall that if $G$ is either compact or abelian, the class of measures $\mu\in M(G)$ satisfying $\lambda(\mu)\in C_{r}^{\ast}(G)$ coincides with the class of Rajchman measures. For the case $q:\mathbb{R}\rightarrow\mathbb{T}$, suppose that $E\subseteq\mathbb{T}$ is a closed set and that $q^{-1}(E)$ is a $M_{0}$-set. Thus, there is a Rajchman measure $\nu\neq 0$ supported on $q^{-1}(E)$. Consider the pushforward measure $\mu:=q_{\ast}\nu$ on $\mathbb{T}$, which is a non-zero measure on $\mathbb{T}$ supported on $E$. Furthermore, $\mu$ is Rajchman. Indeed, this follows from the observation that if $f$ is a function on $\mathbb{T}$ and $g$ is its 2$\pi$-periodic extension as a function on $\mathbb{R}$, then $\int_{\mathbb{R}}g(x)d\nu(x)=\int_{\mathbb{T}}f(t)d\mu(t)$. In particular, $\hat{\nu}(n)=\hat{\mu}(n)$. Hence, since $\hat{\nu}\in C_{0}(\mathbb{R})$ it follows that $\lim_{|n|\rightarrow\infty}\hat{\mu}(n)=0$ and $\mu$ is Rajchman. It follows that $E$ is a $M_0$-set. In order to use this idea within a more general setting, we establish some notation.\\ \\ 
\noindent Let $G$ be locally compact and $H\subseteq G$ be a closed normal subgroup. Given $f\in C_{c}(G)$, we define : $$\phi(f)([x]):=\int_{H}f(xh)dh$$ One can prove that $\phi(f)\in C_{c}(G/H)$ and that in fact, $\phi :C_{c}(G)\rightarrow C_{c}(G/H)$ is surjective (cf. Proposition 1.3.7, in \cite{17}). The Haar measures on $G$ and on $H$ can be normalized in such way that the Weil's formula holds for all $f\in C_{c}(G)$ : $$\int_{G}f(x)dx = \int_{G/H}(\int_{H}f(xh)dh)d\mu[x]$$ where $\mu$ is a $G$-invariant measure on $G/H$ (since $H$ is closed normal, in this case $\mu$ is a Haar measure on $G/H$). One can verify that given $f\in L^{1}(G)$, Weil's formula still holds and that $||\phi(f)||_{1}\leq ||f||_{1}$ and thus, $\phi$ is actually a $^{\ast}$-homomorphism from $L^{1}(G)$ onto $L^{1}(G/H)$. \\ \\ \noindent Every unitary representation of $G$ determines a non degenerate $^{\ast}$-representation $\tilde{\pi}$ of $L^{1}(G)$ as follows : $$\langle\tilde{\pi}(f)\xi,\eta\rangle = \int_{G}f(x)\langle \pi(x)\xi,\eta\rangle dx, \ \text{for $f\in L^{1}(G)$ and $\xi,\eta\in\mathcal{H}(\pi)$}$$ For any representation $\pi$ of $G/H$ and $\xi,\eta\in\mathcal{H}(\pi)$, one shows that : $$\langle \pi\circ q(f)\xi,\eta\rangle =\langle \pi(\phi(f))\xi,\eta\rangle, \ \text{for $f\in L^{1}(G)$}$$ where $q:G\rightarrow G/H$ is the quotient map. This implies that $||\phi(f)||_{C^{\ast}(G/H)}\leq ||f||_{C^{\ast}(G)}$ and thus, $\phi$ can be extended to a $^{\ast}$-homomorphism from $C^{\ast}(G)$ onto $C^{\ast}(G/H)$. For more details, the interested reader is referred to \cite{17}.

\begin{thm}\label{quotientM0} Let $G$ be a locally compact second countable abelian (or compact) group, $H\subseteq G$ a normal closed subgroup and $E\subseteq G/H$ be a closed set. Then, if $q^{-1}(E)$ is a $M_0$-set, so is $E$.\end{thm}
\begin{proof} Since $q^{-1}(E)$ is a $M_0$-set, there is a measure $\mu\in M(q^{-1}(E))$ such that $\lambda_{G}(\mu)\neq 0$ and $\lambda_{G}(\mu)\in C^{\ast}(G)$. Recall that if $\lambda_{G}(\mu)\in C^{\ast}(G)$, then $\mu$ is Rajchman. Since $\lambda_{G}$ is faithful, it follows that $\mu$ is a non-zero Rajchman measure and we can, if needed after considering $|\mu|$, assume that $\mu$ is a non-zero positive Rajchman measure.\footnote{Let $G$ be a locally compact group. Then, if $\mu\in M(G)$ is Rajchman, so is $|\mu|$ (cf. Corollary 12 in \cite{18}).} Hence, it follows that the pushforward measure $q_{\ast}\mu\in M(E)$ is non-zero and thus,  $\lambda_{G/H}(q_{\ast}\mu)\neq 0$. It then suffices to prove that $\lambda_{G/H}(q_{\ast}\mu)\in C^{\ast}(G/H)$ and we can conclude that $E$ is a $M_0$-set. Indeed : \begin{align*} \phi(\lambda_{G}(\mu))\lambda_{G/H}(\phi(f))&= \phi(\lambda_{G}(\mu)\lambda_{G}(f))= \phi(\lambda_{G}(\mu\ast f))\\ & = \lambda_{G/H}(\phi(\mu\ast f))=(\lambda_{G/H}\circ q)(\mu\ast f)\\ & = (\lambda_{G/H}\circ q)(\mu)(\lambda_{G/H}\circ q)(f)\end{align*} Since $\pi\circ q(\mu)=\pi(q_{\ast}\mu)$ for any representation $\pi$ we conclude that : $$\phi(\lambda_{G}(\mu))=\lambda_{G/H}\circ q(\mu)=\lambda_{G/H}(q_{\ast}\mu)\in C^{\ast}(G/H)=C_{r}^{\ast}(G/H)\footnote{Recall that the quotient of an amenable group by a normal subgroup, is an amenable group again.}$$ \end{proof}

\noindent In order to prove a converse to Theorem \ref{quotientM0} we rely on a certain operator that allows  going from $C_{r}^{\ast}(G/H)$ to $C_{r}^{\ast}(G)$ while respecting the support in a convenient way. \\ \\
\noindent Let $G$ be a locally compact group, $\theta\in A(G)\cap C_{c}(G)$ and $T\in VN(G/H)$, where $H$ is a closed normal subgroup of $G$. Then, one can show that the following functional is bounded (Theorem 3.7 in \cite{6}) : $$A(G)\ni u\mapsto\langle T,\phi(\theta u)\rangle$$ Thus, there is an operator $\Phi_{\theta}(T)\in VN(G)$ such that : $$\langle \Phi_{\theta}(T),u\rangle = \langle T,\phi(\theta u)\rangle \ \ ,u\in A(G)$$ \begin{lemma}\label{capital} Let $\theta\in A(G)\cap C_{c}(G)$. Then, the following holds : \begin{enumerate}[(i)]\item{$\Phi_{\theta}$ maps $C_{r}^{\ast}(G/H)$ into $C_{r}^{\ast}(G)$.}\item{If $T\in C_{r}^{\ast}(G/H)\cap J(E)^{\perp}$, then $\Phi_{\theta}(T)\in C_{r}^{\ast}(G)\cap J(q^{-1}(E))^{\perp}$, where $q:G\rightarrow G/H$ is the quotient map.}\end{enumerate}\end{lemma}\begin{proof} The reader can find a proof in \cite{6} (Theorem 3.7).\end{proof}

\noindent Motivated by the definition of property \emph{(l)} in \cite{6}, we introduce the following property : \begin{defn} Let $G$ be a locally compact group and $H\subseteq G$ be a closed normal subgroup. For any $\theta\in C_{c}(G)$, let $|\theta(x)|:=|\theta(x)|$. We say that $H$ has the property $(|l|^{2})$ if for every proper compact $K\subseteq G/H$ there exists $\theta\in A(G)\cap C_{c}(G)$ such that $\phi(|\theta|)=c_{1}$ and $\phi(\theta^{2})=c_{2}$, for positive real constants $c_{1},c_{2}$ on a neighborhood of $K$.\end{defn} \begin{remar}\label{exemplo} As an important example, $\mathbb{Z}\subseteq\mathbb{R}$ has the property $(|l|^{2})$. Indeed, after identifying $\mathbb{T}$ with $[0,1)$ let's fix a proper compact $K\subseteq\mathbb{T}$. We may assume that $0\notin K$ - otherwise replacing $K$ with a suitable translation. Using the regularity of $A(\mathbb{R})$, let $\theta\in A(\mathbb{R})$ such that $\theta(x)=1$ on a neighborhood $\mathcal{U}$ of $K$ and $\theta(x)=0$ whenever $x\notin (0,1)$. If $x\in\mathcal{U}$, note that : $$\phi(|\theta|)([x])=\sum_{h\in\mathbb{Z}}|\theta(x+h)|=\sum_{h\in\mathbb{Z}}\theta(x+h)\theta(x+h)=\phi(\theta^{2})([x])=1$$ An entirely analogous argument shows that $\mathbb{Z}^{n}\subseteq\mathbb{R}^{n}$ and $\mathbb{Z}^{n}\times\lbrace 1\rbrace^{m}\subseteq\mathbb{R}^{n}\times\mathbb{T}^{m}$ also have the property $(|l|^{2})$.\end{remar}

\begin{thm}\label{quotientMO(2)} Let $G$ be a locally compact group, $H\subseteq G$ a closed normal subgroup with property $(|l|^{2})$ and $E\subseteq G/H$ a compact set. Furthermore, assume that $G/H$ is abelian and second countable. Let $q:G\rightarrow G/H$ denote the quotient map and suppose that $E$ is a $M_0$-set. Then, $q^{-1}(E)$ is also a $M_0$-set.\end{thm}
\begin{proof} Since $E\subseteq G/H$ is a $M_0$-set, there is a non-zero Rajchman measure $\mu'\in M(E)$. Moreover, since $\text{supp}(|\mu'|)\subseteq E$ and $|\mu'|$ is also Rajchman, we may assume that there is a non-zero positive Rajchman measure $\mu$ supported by $E$. Since $G/H$ is second countable and $\mu\neq 0$ one has that $\mu(E)>0$. \\ For any $\theta\in A(G)\cap C_{c}(G)$, $\Phi_{\theta}(\lambda(\mu))\in C_{r}^{\ast}(G)$ by Lemma \ref{capital}. Thus, if one shows that $\Phi_{\theta}(\lambda(\mu))=\lambda(\nu)$ for some $\nu\in M(q^{-1}(E))$ such that $\lambda(\nu)\neq 0$, it follows that $q^{-1}(E)$ is a $M_{0}$-set.\\ In order to obtain such measure $\nu$, appealing to Riesz-Markov-Kakutani Representation Theorem it is enough to choose $\theta$ such that the following linear functional $\Psi$ is bounded : $$C_{0}(G)\ni u\mapsto \int_{G/H}(\int_{H}\theta(xh)u(xh)dh)d\mu(x)$$ Indeed, if $\Psi$ is bounded then there is some measure $\nu\in M(G)$ such that : $$\Psi(u)=\int_{G}u(s)d\nu(s)$$ By construction, for $u\in A(G)$ we have that : $$\langle\Phi_{\theta}(\lambda(\mu)),u\rangle =\langle\lambda(\mu),\phi(\theta u)\rangle=\Psi(u)=\langle\lambda(\nu),u\rangle$$ and we can conclude that $\Phi_{\theta}(\lambda(\mu))=\lambda(\nu)$. Thus, it remains to choose an appropriate $\theta$. Since $E$ is compact and $H\subseteq G$ has the property $(|l|^{2})$, let $\theta\in A(G)\cap C_{c}(G)$ be such that $\phi(|\theta|)=c_{1}>0$ and $\phi(\theta^{2})=c_{2}>0$ on a neighborhood of $E$. It follows that : \begin{align*} |\int_{G/H}(\int_{H}\theta(xh)u(xh)dh)d\mu(x)|&\leq \int_{G/H}|\int_{H}\theta(xh)u(xh)dh|d\mu(x)\\ & \leq||u||_{\infty}\int_{G/H}|\phi(\theta)(x)|d\mu(x)\\ &\leq ||u||_{\infty}c_{1}\mu(E)\end{align*} Hence, $\Psi$ is bounded. Moreover, it follows from Lemma \ref{capital} that $\nu\in M(q^{-1}(E))$. It remains to check that indeed $\lambda(\nu)\neq 0$. For this purpose, it suffices to choose $u\in A(G)$ such that $\Psi(u)\neq 0$. We let $u=\theta$ and note that : $$\int_{G/H}(\int_{H}\theta(xh)\theta(xh)dh)d\mu(x) = \int_{G/H}\phi(\theta^{2})(x)d\mu(x) = c_{2}\mu(E) >0 $$\end{proof}

\begin{corol}\label{quocienteM0} Let $q:\mathbb{R}^{n}\times\mathbb{T}^{m}\rightarrow\mathbb{T}^{n+m}$ be the quotient map for integers $n,m\geq 0$ and $E\subseteq\mathbb{T}^{n+m}$ be a closed subset. Then, $E$ is a $M_{0}$-set if and only if $q^{-1}(E)$ is a $M_{0}$-set.\end{corol}
\begin{proof} The result follows from Theorem \ref{quotientM0} and Theorem \ref{quotientMO(2)}. \end{proof}

\noindent We finish the section with results concerning the preservation of operator $M_0$-sets and $M_0$-sets under finite direct products. It will be useful to consider the left and right slice maps. Given $\omega\in(\mathcal{K}(H_2,K_2))^{\ast}=\mathcal{C}_{1}(K_2,H_2)$, the left slice map $L_{\omega}:\mathcal{K}(H_1\otimes H_2, K_1\otimes K_2)\rightarrow\mathcal{K}(H_1,K_1)$ is defined on elementary tensors by : $$L_{\omega}(A\otimes B)=\omega(B)A$$ A useful property of $L_{\omega}$ is that if $T\in\mathcal{K}(H_1\otimes H_2,K_1\otimes K_2)$ is supported on $\rho(\kappa_1\times\kappa_2)$ then, $\text{supp}(L_{\omega}(T))\subseteq\kappa_{1}$ (cf. \cite{2}). The right slice operator $R_{\omega}:\mathcal{K}(H_1\otimes H_2,K_1\otimes K_2)\rightarrow\mathcal{K}(H_2,K_2)$ is defined analogously.

\begin{thm}\label{preservaprod} Let $(X_i,\nu_i),(Y_i,\mu_i)$ be standard measure spaces and $\kappa_{i}\subseteq X_{i}\times Y_{i}$ be $\omega$-closed sets, for $i=1,2$. The set $\rho(\kappa_{1}\times\kappa_{2})$ is an operator $M_{0}$-set if and only if both $\kappa_{1}$ and $\kappa_{2}$ are operator $M_{0}$-sets.
\end{thm}
\begin{proof} Let $\kappa_{1}$ and $\kappa_{2}$ be operator $M_{0}$-sets, so that there are non zero measures $\sigma_{1}\in\mathbb{A}(\hat{\kappa_{1}})$ and $\sigma_{2}\in\mathbb{A}(\hat{\kappa_{2}})$ such that $T_{\sigma_{1}}\in\mathcal{K}(H_1,K_1)$ and $T_{\sigma_{2}}\in\mathcal{K}(H_2,K_2)$. Following the proof of Theorem 3.8 in \cite{2}, we conclude that $T_{\sigma}:=T_{\sigma_{1}}\otimes T_{\sigma_{2}}\in\mathcal{K}(L^{2}(X_{1}\times X_{2}),L^{2}(Y_{1}\times Y_{2}))$ for a non zero Arveson measure $\sigma$ supported in $\widehat{\rho(\kappa_{1}\times\kappa_{2})}$. Hence, $\rho(\kappa_{1}\times\kappa_{2})$ is an operator $M_{0}$-set.\\ \\ Conversely, suppose that $\rho(\kappa_{1}\times\kappa_{2})$ is an operator $M_{0}$-set so that there is a non zero measure $\sigma\in\mathbb{A}(\widehat{\rho(\kappa_{1}\times\kappa_{2})})$ and $T_{\sigma}\in\mathcal{K}(H_{1}\otimes H_2,K_{1}\otimes K_{2})$. Note that $L_{\omega}(T_{\sigma})\in\mathcal{K}(H_{1},K_{1})$ and supp$(L_{\omega}(T_{\sigma}))\subseteq\kappa_{1}$. Thus, it suffices to prove that for some $\omega$, $L_{\omega}(T_{\sigma})$ is a pseudo-integral operator - say $T_{\gamma}$ - for a non-zero Arveson measure $\gamma$. It will follow from Theorem \ref{arvesonop} that $\gamma\in\mathbb{A}(\hat{\kappa_{1}})$ and thus, $\kappa_{1}$ is an operator $M_{0}$-set. The case for $\kappa_{2}$ is entirely analogous, using the right slice operator $R_{\omega}$ instead. \\ We let $\Omega = Y_1\times Y_2\times X_1\times X_2$, $\pi$, $\pi_{X_i}$, $\pi_{Y_i}$ be respectively the projections of $\Omega$ onto $Y_{1}\times X_1$, $X_i$ and $Y_i$ and $f_{i}\in L^{2}(X_i,\nu_i)$, $g_{i}\in L^{2}(Y_i,\mu_{i})$, for $i=1,2$ with $\omega=f_{2}\otimes g_{2}$. \\ Note that if $\sigma$ is Arveson (in particular with finite total variation), then so is $|\sigma|$. Consequently, we can assume that $\sigma\in\mathbb{A}(\widehat{\rho(\kappa_1\times \kappa_2)})$ is finite, non-zero and non-negative and as $\Omega$ is completely metrizable and separable, it follows that $\sigma$ is Radon.\footnote{Let $X$ be a completely metrizable and separable space and $\mu$ a (positive) Borel measure on $X$. Then, $X$ is Radon (cf. \cite{19}, Theorem 7.1.7). In fact, if $w(X)$ is the weight of $X$, then there is a non-tight finite (positive) Borel measure on $X$ if and only if $w(X)$ is a measurable cardinal (cf. \cite{20}, Theorem 438H). Existence of measurable cardinals is not provable within ZFC. } Since $\sigma\neq 0$ and is Radon, there is a compact $K\subseteq\Omega$ with $\sigma(K)>0$. Note that $\pi_{X_2}(K):=C\subseteq X_2$ and $\pi_{Y_2}(K):=C'\subseteq Y_2$ are compact sets and let $\mathcal{U},\mathcal{V}$ be open sets such that $C\subseteq\mathcal{U}$ and $C'\subseteq\mathcal{V}$. Since $X_2$ and $Y_2$ are normal, there are continuous functions $f_{2}:X_{2}\rightarrow [0,1]$ and $g_{2}:Y_2\rightarrow [0,1]$ such that $\chi_{C}\leq f_2\leq \chi_{\mathcal{U}}$ and $\chi_{C'}\leq g_2\leq \chi_{\mathcal{V}}$. Since $\nu_2$ is Radon, then : $$\inf\lbrace\nu_{2}(\mathcal{W}):C\subseteq\mathcal{W}\ \text{and $\mathcal{W}$ is open}\rbrace = \nu_{2}(C)<\infty$$ and we can choose $\mathcal{V}$ such that $\nu_{2}(\mathcal{V})<\infty$. Similarly, we can choose $\mathcal{W}$ such that $\mu_{2}(\mathcal{W})<\infty$. Consequently : $$\int_{X_2}f_{2}^{2}(x_2)d\nu_{2}\leq\int_{X_{2}}\chi_{\mathcal{U}}d\nu_{2}=\nu_{2}(\mathcal{U})<\infty$$ and we conclude that $f_{2}\in L^{2}(X_2,\nu_2)$. Similarly, we conclude that $g_{2}\in L^{2}(Y_2,\mu_2)$. This is our choice of $f_2$ and $g_2$, defining $\omega$. Now note that : $$(L_{\omega}(T_{\sigma})f_{1},\overline{g_1})=\langle L_{\omega}(T_{\sigma}),f_{1}\otimes g_{1}\rangle = \langle T_{\sigma},(f_{1}\otimes g_{1})\otimes\omega\rangle = (T_{\sigma}(f_{1}\otimes f_{2}),\overline{g_{1}\otimes g_{2}})$$ Since $\sigma$ is an Arveson measure on $\Omega$ we have : $$(L_{\omega}(T_{\sigma})f_1,g_1)=\int_{\Omega}f_{1}(x_1)f_2(x_2)\overline{g_1(y_1)}\overline{g_2(y_2)}d\sigma((y_1,y_2),(x_1,x_2))$$ Letting $y=(y_1,y_2)$ and $x=(x_1,x_2)$, we define a measure $\gamma$ on $Y_{1}\times X_{1}$ as follows : $$\gamma(E):=\int_{\Omega}\chi_{\pi^{-1}(E)}((y,x))f_{2}(x_2)\overline{g_{2}(y_2)}d\sigma(y,x)$$ Therefore, we may conclude that : $$(L_{\omega}(T_{\sigma})f_1,g_1)=\int_{Y_{1}\times X_{1}}f_{1}(x_1)\overline{g_1(y_1)}d\gamma(y_1,x_1)$$ If one proves that $\gamma$ an Arveson measure, we may conclude that $L_{\omega}(T_{\sigma})$ is pseudo-integral. It will then be enough to check that $\gamma$ is non-zero and we are done. On one hand : $$\gamma(Y_1\times X_1)=\int_{\Omega}\chi_{\Omega}(y,x)f_{2}(x_2)g_{2}(y_2)d\sigma(y,x)\leq\int_{\Omega}\chi_{\Omega}(y,x)d\sigma <\infty$$ and thus, $\gamma$ has finite total variation. On the other hand, let $E=\alpha\times X_1$, with a measurable set $\alpha\subseteq Y_1$. Then, we have that : \begin{align*}|\gamma|(\alpha\times X_1)&\leq\int_{\Omega}\chi_{\alpha\times Y_{2}\times X_{1}\times X_2}((y,x))f_2(y_2)g_2(x_2)d|\sigma|((y,x))\\&=\int_{\Omega}\chi_{\alpha}(y_1)\chi_{Y_2\times X_{1}\times X_2}(y_2,x_1,x_2)f_{2}(y_2)g_{2}(x_2)d|\sigma|((y,x))\\&\leq (\int_{\Omega}f_{2}^{2}g_{2}^{2}d|\sigma|)^{\frac{1}{2}}(\int_{\Omega}\chi_{\alpha}\chi_{Y_{2}\times X_{1}\times X_{2}}d|\sigma|)^{\frac{1}{2}}:=A^{\frac{1}{2}}B^{\frac{1}{2}}\\&\leq k\nu_{1}(\alpha), \text{for some finite positive $k$. In fact :}\end{align*} Since $\sigma$ is Arveson :  $$A^{\frac{1}{2}}\leq c (\int_{X_2} f_{2}^2(x_2)d\nu_2)^{\frac{1}{2}}(\int_{Y_2}g_{2}^{2}(y_2)d\mu_{2})^{\frac{1}{2}}=c||f_2||||g_2||<\infty$$ Furthermore, note that given measurable spaces $(X_i,\nu_i)$ and $(Y_i,\mu_i)$, a measure $\sigma$ on $(X_1\times X_2)\times(Y_1\times Y_2)$ can be identified with a measure $\tilde{\sigma}$ on $X_{1}\times(X_{2}\times Y_{1}\times Y_{2})$ as the product $\sigma$-algebras can be canonically identified. In this way, the marginal measures $|\tilde{\sigma}|(\alpha)$ and $|\sigma|(\alpha\times X_2)$ - for $\alpha\subseteq X_1$ measurable - coincide. If $\sigma$ is Arveson, it then follows that there is some positive finite constant $d$ such that $|\tilde{\sigma}(\alpha)|\leq d a_{1}^{2}(\alpha)$. Thus, $B^{\frac{1}{2}}\leq (d\nu_{1}^{2}(\alpha))^{\frac{1}{2}}$.\\ It remains to check that $\gamma$ is non-zero : let $E=\pi_{Y_1}(K)\times\pi_{X_1}(K)$ which is compact, hence closed and thus Borel subset of $Y_1\times X_1$. Then : $$\gamma(E)=\int_{\Omega}\chi_{\pi^{-1}(E)}(y,x)f_{2}(x_2)g_{2}(y_2)d\sigma(y,x)\geq \sigma(K)>0$$ since if $(y_1,y_2,x_1,x_2)\in K$, then $\chi_{\pi^{-1}(E)}=1$ and $f_2(x_2)=g_2(y_2)=1$, by construction.\end{proof}

\noindent We can thus establish the corresponding fact for $M_0$-sets :

\begin{corol}\label{prodM0} Let $G_1$ and $G_2$ be locally compact second countable groups and $E_{1}\subseteq G_{1}$ and $E_{2}\subseteq G_{2}$ be closed sets. Then : \begin{enumerate}[(i)]\item{If $E_1$ and $E_2$ are sets of restricted multiplicity, then $E_{1}\times E_{2}$ is a set of restricted multiplicity.}\item{If $E_1$ and $E_2$ are sets of extended uniqueness, then $E_{1}\times E_{2}$ is a set of extended uniqueness.}\end{enumerate}\end{corol}
\begin{proof} Let $E_1$ and $E_2$ be $M_0$-sets. By Theorem \ref{operatorM0versusM0}, $E_{1}^{\ast}$ and $E_{2}^{\ast}$ are operator $M_{0}$-sets. Since $\rho(E_{1}^{\ast}\times E_{2}^{\ast})=(E_{1}\times E_{2})^{\ast}$ it follows by Theorem \ref{preservaprod} and Theorem \ref{operatorM0versusM0} that $E_{1}\times E_{2}$ is a $M_0$-set. The proof is entirely analogous for the case of sets of extended uniqueness.\end{proof}

\section{Complexity of sets of uniqueness and extended uniqueness}

\noindent In section 4.1 we follow \cite{3} and establish the $\Pi_{1}^{1}$-hardness of $\mathcal{U}(\mathbb{T})$ and $\mathcal{U}_{0}(\mathbb{T})$. In section 4.2, we prove the coanalycity of the collection of closed sets of uniqueness and of the collection of closed sets of extended uniqueness (with respect to the Effros Borel space) in greater generality. Combining these results with the $\Pi_{1}^{1}$-hardness of $\mathcal{U}(\mathbb{T})$ (and $\mathcal{U}_{0}(\mathbb{T})$) and using the preservation properties proven in section 3.2, we locate the complexity of the following sets : \\ \\ $\bullet$ $\mathcal{U}(G)$ and $\mathcal{U}_{0}(G)$ are coanalytic if $G$ is locally compact, Hausdorff and second countable. \\ \\$\bullet$ $\mathcal{U}(G)$ is $\Pi_{1}^{1}$-complete for any connected and locally compact Lie group $G$. \\ \\$\bullet$ $\mathcal{U}_{0}(G)$ is $\Pi_{1}^{1}$-complete if $G$ is a connected abelian Lie group. 

\subsection{The case of $\mathbb{T}$}

\noindent In order to prove the $\Pi_{1}^{1}$-hardness of $\mathcal{U}(\mathbb{T})$ we use the Salem-Zygmund Theorem, which is a remarkable fact relating sets of uniqueness of the circle with Pisot numbers. \\ \\ Consider an interval $[a,b]$ with length $l=b-a$ and a sequence of positive real parameters : $$0=\eta_{0}<\eta_{1}<...<\eta_{k}<\eta_{k+1}=1$$ such that for $i<k$, $\xi:=1-\eta_{k}<\eta_{i+1}-\eta_{i}$. Moreover, let $E$ be the following disjoint union of closed sets : $$\bigcup_{i=0}^{k}[a+l\eta_{i},a+l\eta_{i}+l\xi]$$ We say that $E$ was obtained from $[a,b]$ by a dissection of type $(\xi,\eta_{1},...,\eta_{k})$.\\ We now start from $[a,b]=[0,2\pi]$ and apply iteratively this operation of dissection. As result, we have a decreasing sequence of closed sets $E_{1}\supseteq ...\supseteq E_{n}\supseteq E_{n+1}\supseteq ...$ and define : $$E(\xi,\eta_{1},...,\eta_{k})=\bigcap_{n\in\omega}E_{n}$$ It is clear that $E$ is a non-empty perfect set. For instance, the usual Cantor set in this notation is simply the set $E(\frac{1}{3},\frac{2}{3})$. The Salem-Zygmund Theorem gives us a full characterization of these type of sets in terms of sets of uniqueness : 
\begin{thm}\label{salem}$E(\xi,\eta_{1},...,\eta_{k})$ is a set of uniqueness if and only if $\frac{1}{\xi}$ is a Pisot number\footnote{An algebraic integer $x$ is called a Pisot number if $x>1$ and all its conjugates have absolute value smaller than 1.} and each $\eta_{i}\in\mathbb{Q}(\frac{1}{\xi})$.\end{thm}
\begin{proof} The reader can find a proof in \cite{3} (Theorem III.3.1).\end{proof}

\noindent Another remarkable fact about sets of uniqueness of $\mathbb{T}$ which is needed is Bary's Theorem, which states an important closure property of $\mathcal{U}(\mathcal{K}(\mathbb{T}))$ :
\begin{thm}\label{bary} The union of countably many closed sets of uniqueness of $\mathbb{T}$ is also a set of uniqueness.\end{thm}
\begin{proof} The reader can find a proof in \cite{3} (Theorem I.5.1).\footnote{Let $\kappa<2^{\aleph_{0}}$ be a cardinal. Under the assumption $MA(\kappa)+\neg CH$ one can prove that in fact the union of $\kappa$ many closed sets of uniqueness of $\mathbb{T}$ is also a set of uniqueness. Since $MA(\aleph_{0})$ holds in ZFC, this includes the case of Theorem \ref{bary}. The reader can find a proof in \cite{10} (Theorem 4.86).}\end{proof}

\noindent We can now establish $\Pi_{1}^{1}$-hardness for $\mathcal{U}(\mathbb{T})$ :

\begin{thm}\label{d=1} The set of closed sets of uniqueness of $\mathbb{T}$ is $\Pi_{1}^{1}(\mathcal{K}(\mathbb{T}))$-hard.\end{thm}
\begin{proof} By Theorem \ref{redu} it is enough to define a continuous map $F:\mathcal{K}([0,1])\rightarrow\mathcal{K}(\mathbb{T})$ such that $\mathcal{K}(\mathbb{Q}')=F^{-1}(\mathcal{U}(\mathbb{T}))$. We start by defining a continuous map $f:[0,1]\rightarrow\mathcal{K}(\mathbb{T})$ as follows : $$x\mapsto E(\frac{1}{4},\frac{3}{8}+\frac{x}{9},\frac{3}{4})$$It follows by Theorem \ref{salem} that $\mathbb{Q}=f^{-1}(\mathcal{U}(\mathcal{K}(\mathbb{T})))$. Now, we define a map $F:\mathcal{K}([0,1])\rightarrow\mathcal{K}(\mathbb{T})$ as follows : $$K\mapsto \bigcup_{x\in K}f(x)$$ By the properties of the Vietoris topology, $F$ is continuous. Moreover, by Theorem \ref{bary}, $F^{-1}(\mathcal{U}(\mathbb{T}))=\mathcal{K}(\mathbb{Q}')$ as we wanted.\end{proof}

\begin{remar} The proof of Theorem \ref{salem} given in \cite{3} actually shows a stronger statement : If $\frac{1}{\xi}$ is not a Pisot number or some $\eta_{i}\notin\mathbb{Q}(\frac{1}{\xi})$, then $E(\xi,\eta_{1},...,\eta_{k})$ is a set of restricted multiplicity (Theorem III.4.4, in \cite{3}). Since sets of uniqueness are also sets of extended uniqueness, the proof of Theorem \ref{d=1} also shows the following :\end{remar}

\begin{thm}\label{hardMO} The set of closed sets of extended uniqueness of $\mathbb{T}$ is $\Pi_{1}^{1}(\mathcal{K}(\mathbb{T}))$-hard.\end{thm}

\subsection{General case}

\noindent We start this last section by identifying on which Polish spaces we will work on. \\ \\
\noindent $\bullet$ For a locally compact space $X$, we will consider the Fell topology on $\mathcal{F}(X)$, which is compact metrizable, induces the Effros Borel space and coincides with the Vietoris topology whenever $X$ is compact. \\ \\ $\bullet$ Let $X$ be a locally compact, Hausdorff and second countable space. Then, $C_0(X)$ is separable\footnote{Since $X$ is second countable, so it is its one-point compactification $\tilde{X}$. Thus, $C(\tilde{X})$ is separable as $\tilde{X}$ is compact, Hausdorff and second countable. Since subspaces of separable metric spaces are separable, we are done} and since $M(X)=C_{0}(X)^{\ast}$, it follows that $B_{1}(M(G))$ is $w^{\ast}$-metrizable for a locally compact and second countable group $G$ (we adopt the convention that topological groups are Hausdorff). Furthermore, by Banach-Alaoglu Theorem, it follows that : $$G\ \text{locally compact second countable group, then}\ (B_{1}(M(G)),w^{\ast})\ \text{is Polish}$$
$\bullet$ Let $G$ be locally compact and second countable. Any element of $u\in A(G)$ is of the form $u= f\ast\check{g}$ for $f,g\in L^{2}(G)$ (cf. Theorem 2.4.3, in \cite{17}) and with norm $||u||=\inf\lbrace||f||_{2}||g||_{2}\rbrace$ with infimum taken over all representations of the form $u=f\ast\check{g}$. Since $G$ is second countable, then $L^{2}(G)$ is separable and consequently, $A(G)$ is also separable. Thus, and since $VN(G)=A(G)^{\ast}$, $B_{1}(VN(G))$ is $w^{\ast}$-metrizable. Similarly as before, we conclude that : $$G\ \text{locally compact second countable group, then}\ (B_{1}(VN(G)),w^{\ast})\ \text{is Polish}$$ $\bullet$ Let $G$ be locally compact and second countable. It follows that $L^{1}(G)$ is separable. Let $C^{\ast}(G)$ be the completion of $L^{1}(G)$ under the norm $||.||_{C^{\ast}(G)}$. It is well known that $||f||_{C^{\ast}(G)}\leq ||f||_{1}$ and thus, $C^{\ast}(G)$ is separable. Since it is defined as a completion, it follows that $C^{\ast}(G)$ is Polish. \\ \\ $\bullet$ We recall that if $G$ is locally compact, for any $\varphi\in A(G)^{\ast}$ there is an unique operator $T_{\varphi}\in VN(G)$ such that : $$\langle T_{\varphi}(f),g\rangle_{2}=\langle\varphi,\overline{g}\ast \check{f}\rangle =\langle\varphi,\check{f\ast\tilde{g}}\rangle, \ \text{for $f,g\in L^{2}(G)$} $$ The assignment $\varphi\mapsto T_{\varphi}$ is then a surjective linear isometry from $A(G)^{\ast}$ to $VN(G)$ and a homeomorphism if one considers the $w^{\ast}$ topology on both the domain and codomain. Furthermore, if $\varphi\in A(G)^{\ast}$ is given by $\langle\varphi_{\mu},u\rangle =\int_{G}u(s)d\mu(s)$ for some $\mu\in M(G)$, we have that $T_{\varphi_{\mu}}$ actually coincides with $\lambda(\mu)$. For a proof of this collection of facts, the reader is referred to \cite{17} (Theorem 2.3.9). \\ In the remaining of this section, we consider the representation $\lambda$ of $M(G)$ restricted to the unit ball and with the following topologies : $$\lambda : (B_{1}(M(G)),w^{\ast})\rightarrow (VN(G),w^{\ast})$$ It follows from previous observations that $\lambda$ is continuous. Indeed, since $B_{1}(M(G))$ is $w^{\ast}$-metrizable, it is sufficient to verify sequential continuity : if $\mu_{n}\rightarrow\mu$, it follows by definition that $\varphi_{\mu_{n}}\rightarrow\varphi_{\mu}$ and thus, $\lambda(\mu_{n})\rightarrow\lambda(\mu)$. Note that in fact, the definition of the pairing $A(G)^{\ast}=VN(G)$ entails the continuity of $\lambda$ even if one considers the WOT topology in the codomain.

\subsubsection{$M$-sets}

\begin{thm}\label{coanalM} Let $G$ be a locally compact and second countable group. Then, $\mathcal{U}(G)$ is coanalytic.\end{thm}\begin{proof} \textbf{Step 1} : Consider $C^{\ast}(G)$ and $C^{\ast}_{r}(G)$ respectively as the completions of $L^{1}(G)$ under the norm $||.||_{\ast}$ and $||.||_{\lambda}$ (as usual, $||f||_{\lambda}:=||\lambda (f)||$). Since $||.||_{\lambda}\leq ||.||_{\ast}$ this entails a continuous surjective map, $\Upsilon : C^{\ast}(G)\rightarrow C^{\ast}_{r}(G)\subseteq\mathcal{L}(L^{2}(G))$, taking the operator norm in the codomain. Note that $C:= B_{1}(VN(G))\cap C_{r}^{\ast}(G)$ is closed in the operator norm and thus, since $\Upsilon$ is continuous, $P:=\Upsilon^{-1}(C)$ is also closed. As a closed subspace of a Polish space, it follows that $P$ is also Polish. Let's denote the restriction of $\Upsilon$ to $P$ still as $\Upsilon$. Then, $\Upsilon : P\rightarrow B_{1}(VN(G))$ is a function whose image is $B_{1}(VN(G))\cap C^{\ast}_{r}(G)$, as $\Upsilon$ is surjective. Since $\Upsilon$ was continuous with respect to the operator norm, it will remain continuous with respect to the $w^{\ast}$-topology in the range. Thus : $$\Gamma_{1} := C_{r}^{\ast}(G)\cap B_{1}(VN(G))\subseteq (B_{1}(VN(G)),w^{\ast})\ \text{is analytic}$$\textbf{Step 2} Considering $\mathcal{F}(G)$ with the Fell topology and $(B_{1}(VN(G)),w^{\ast})$, the following set is Borel (even closed) : $$\Gamma_2 =\lbrace (S,E)\in B_{1}(VN(G))\times\mathcal{F}(G) : supp^{\ast}(S)\subseteq E\rbrace$$ Let $\lbrace (S_n,E_n)\rbrace\subseteq\Gamma_{2}$ such that $(S_n,E_n)\rightarrow (S,E)$. It is enough to prove that $supp^{\ast}(S)\subseteq E$. We use the characterization of $supp^{\ast}(T)$ given in Lemma \ref{propsup} for $T\in VN(G)$. Let $u\in A(G)\cap C_{c}(G)$ such that $\text{supp}(u)\cap E=\emptyset$ and $\mathcal{W}=\lbrace F\in\mathcal{F}(G):F\cap\text{supp}(u)=\emptyset\rbrace$ which is an open neighborhood of $E$ (in the Fell topology). Since $E_n\rightarrow E$, then $E_{k}\subseteq\mathcal{W}$ for all sufficiently large $k$. Since $supp^{\ast}(S_k)\subseteq E_k$, then for all sufficiently large $k$ one has that $\langle u,S_k\rangle = 0$. On the other hand, since $S_{n}\xrightarrow{w^{\ast}}S$, it follows that $\langle u,S\rangle = 0$ and thus $supp^{\ast}(S)\subseteq E$ as we wanted.\\ \\ \textbf{Step 3}: Define $P\subseteq B_{1}(VN(G))\times\mathcal{F}(G)$ as the following subset : $$P=\lbrace (S,E) : (S,E)\in\Gamma_{2}, S\neq 0\rbrace\cap\pi_{1}^{-1}(\Gamma_{1})$$ where $\pi_{i}$ ($i=1,2$) denotes the projection onto the factors of the product $B_{1}(VN(G))\times\mathcal{F}(G)$. By  Step 1, $\Gamma_{1}$ is analytic and thus so is $\pi_{1}^{-1}(\Gamma_{1})$ since $\pi_1$ is Borel.\footnote{If $X,Y$ are Polish, $f:X\rightarrow Y$ a Borel function, $A\subseteq X$ and $B\subseteq Y$ analytic, then $f(A)$ and $f^{-1}(B)$ are also analytic (cf \cite{3}, Theorem 14.4).} By Step 2, and since the condition $S\neq 0$ is Borel and the intersection of analytic sets is analytic, $P$ is analytic. Again, since $\pi_2$ is Borel, it follows that $\pi_{2}(P)$ is analytic. Note that : $$\pi_{2}(P)=\mathcal{M}(G)$$ and we conclude that $\mathcal{U}(G)$ is coanalytic.\end{proof}

\noindent Given a class $\Gamma$ of sets in a standard Borel space $X$ and $A\subseteq X$, we say that $A$ is Borel $\Gamma$-hard if for every standard Borel space $Y$ and $B\in\Gamma(Y)$, there is a Borel function $f:Y\rightarrow X$ such that $B=f^{-1}(A)$. This definition is very similar to the one of $\Gamma$-hardness, except that we use Borel reductions instead of continuous reductions. In fact, if $X$ is Polish then $A\subseteq X$ is $\Sigma_{1}^{1}$-hard if and only if $A$ is Borel $\Sigma_{1}^{1}$-hard (cf. \cite{3}, pp. 207). Evidently, the same holds for $\Pi_{1}^{1}$. By Theorem \ref{d=1}, $\mathcal{U}(\mathbb{T})\subseteq\mathcal{F}(\mathbb{T})$ is $\Pi_{1}^{1}$-hard and thus, since $\mathcal{F}(\mathbb{T})$ is Polish, it is also Borel $\Pi_{1}^{1}$-hard. By Theorem \ref{hardMO}, the same can be said about $\mathcal{U}_{0}(\mathbb{T})$.\\ \noindent Now let $X$ be locally compact and $A\subseteq X$ such that there is a Borel function $f:\mathcal{F}(\mathbb{T})\rightarrow\mathcal{F}(X)$ such that $f^{-1}(A)=\mathcal{U}(\mathbb{T})$ or $f^{-1}(A)=\mathcal{U}_{0}(\mathbb{T})$. It follows that $A$ is Borel $\Pi_{1}^{1}$-hard and thus, $\Pi_{1}^{1}$-hard.

\begin{thm}\label{M} Let $G$ be a connected locally compact Lie group. Then, $\mathcal{U}(G)$ is $\Pi_{1}^{1}$-complete.\end{thm}
\begin{proof} By Theorem \ref{coanalM}, it is enough to prove that $\mathcal{U}(G)$ is $\Pi_{1}^{1}$-hard. We divide the analysis in two cases : \\ \\ \textbf{Step 1}: Suppose that $G$ is abelian and thus, $G\approx\mathbb{R}^{n}\times\mathbb{T}^{m}$. By Corollary \ref{isoMsets} it suffices to show that $\mathcal{U}(\mathbb{R}^{n}\times\mathbb{T}^{m})$ is $\Pi_{1}^{1}$-hard. \\ Firstly, note that for any positive integer $k$, the function $g:\mathcal{F}(\mathbb{T})\rightarrow\mathcal{F}(\mathbb{T}^{k})$ such that $E\mapsto E^{k}$ is measurable and, by Theorem \ref{prod}, is such that $\mathcal{U}(\mathbb{T})=g^{-1}(\mathcal{U}(\mathbb{T}^{k}))$. Thus, $\mathcal{U}(\mathbb{T}^{k})$ is $\Pi_{1}^{1}$-hard. \\ Now, let $q:\mathbb{R}^{n}\times\mathbb{T}^{m}\rightarrow\mathbb{T}^{n+m}$ be the quotient map and $f:\mathcal{F}(\mathbb{T}^{n+m})\rightarrow\mathcal{F}(\mathbb{R}^{n}\times\mathbb{T}^{m})$ be given by $E\mapsto q^{-1}(E)$. \\ Note that this map is measurable. Indeed, if $A=\lbrace F\in\mathcal{F}(\mathbb{R}^{n}\times\mathbb{T}^{m}):F\cap \mathcal{U}\neq\emptyset\rbrace$, then $f^{-1}(A)=\lbrace F\in\mathcal{F}(\mathbb{T}^{n+m}):F\cap q(\mathcal{U})\neq\emptyset\rbrace$. Since $q$ is the quotient map by an action of a subgroup, $q(\mathcal{U})$ is an open for each open set $\mathcal{U}$ and we may conclude that $f$ is measurable. Furthermore, by Theorem \ref{inv} combined with Theorem \ref{teorema1} one has that $E\in\mathcal{U}(\mathbb{T}^{n+m})$ if and only if $q^{-1}(E)\in\mathcal{U}(\mathbb{R}^{n}\times\mathbb{T}^{m})$ and we are done. \\ \\ \textbf{Step 2}: To consider the general case, note that $G/[G,G]$ is a connected abelian Lie group. By Step 1, $\mathcal{U}(G/[G,G])$ is $\Pi_{1}^{1}$-hard and we use the same argument as before with the quotient map $q:G\rightarrow G/[G,G]$.\end{proof}

\subsubsection{$M_0$-sets}

\begin{thm}\label{coanM0} Let $G$ be a locally compact and second countable group. Then, $\mathcal{U}_{0}(G)$ is coanalytic.\end{thm}
\begin{proof} \textbf{Step 1} : Since the ball $B_{1}(VN(G))$ is $w^{\ast}$-closed by Banach-Alaoglu Theorem, it follows that $P:=\lambda^{-1}(B_{1}(VN(G)))\subseteq (B_{1}(M(G)),w^{\ast})$ is closed and thus, Polish. Consequently, $\lambda(P)\subseteq (B_{1}(VN(G)),w^{\ast})$ is analytic. By the proof of Theorem \ref{coanalM}, $C^{\ast}_{r}(G)\cap B_{1}(VN(G))$ is also analytic. Since intersection of analytic sets is analytic, then the following is analytic in $(B_{1}(VN(G)),w^{\ast})$ : $$A:=\lambda(P)\cap C_{r}^{\ast}(G)\cap B_{1}(VN(G))\setminus\lbrace 0\rbrace$$ It follows that the following is analytic in $(B_{1}(M(G)),w^{\ast})$ : $$\Lambda_{1}:=\lambda^{-1}(A)=\lbrace\mu\in B_{1}(M(G)):\lambda(\mu)\in C_{r}^{\ast}(G)\cap B_{1}(VN(G))\setminus\lbrace 0\rbrace\rbrace$$ \textbf{Step 2} : Consider $(B_{1}(M(G)),w^{\ast})$ and $\mathcal{F}(G)$ with the Fell topology and let : $$\Lambda_{2} =\lbrace (\mu,E): supp^{\ast}(\lambda(\mu))\subseteq E\rbrace$$ where $\pi_{i}$ (for $i=1,2$) is the projection onto the factors of $B_{1}(M(G))\times\mathcal{F}(G)$. The proof that $\Lambda_{2}$ is closed, hence analytic, goes exactly as in the proof that $\Gamma_{2}$ is closed in Theorem \ref{coanalM}. Indeed, as pointed out earlier, $\mu_{n}\xrightarrow{w^{\ast}}\mu$ implies that $\lambda(\mu_{n})\xrightarrow{w^{\ast}}\lambda(\mu)$ and we can reproduce the argument used in the proof of Theorem \ref{coanalM}. \\ \\ \noindent \textbf{Step 3} : Define $P\subseteq B_{1}(M(G))\times\mathcal{F}(G)$ as the following subset : $$P=\lbrace (\mu,E):\mu\in M(E),\ \lambda(\mu)\in C_{r}^{\ast}(G)\cap B_{1}(VN(G)), \ \lambda(\mu)\neq 0\rbrace $$ Since $P=\pi_{1}^{-1}(\Lambda_{1})\cap\Lambda_{2}$ it follows by  Step 1) and Step 2) that $P$ is analytic. Thus, to prove that $\mathcal{U}_{0}(G)$ is coanalytic it is enough to show that $\mathcal{M}_{0}(G)=\pi_{2}(P)$ : \\ \noindent On one hand, if $E\in\pi_{2}(P)$, then there is a measure $\mu\in B_{1}(M(G))$ such that $(\mu,E)\in P$. By definition of $P$, this means that $\lambda(M(E))\cap C_{r}^{\ast}(G)\neq\lbrace 0\rbrace$ and thus, $E$ is a $M_0$-set.\\ Conversely, suppose that $E\subseteq G$ is a $M_0$-set and thus there is some $\mu\in M(E)$ such that $\lambda(\mu)\in C_{r}^{\ast}(G)\setminus\lbrace 0\rbrace$. In particular, $\mu\neq 0$. Then, $\tilde{\mu}=\frac{\mu}{||\mu||}\in B_{1}(M(G))\cap M(E)$ and $T:=\lambda(\tilde{\mu})\in C_{r}^{\ast}(G)\setminus\lbrace 0\rbrace$. If $||T||\leq 1$, then $\lambda(\tilde{\mu})\in C_{r}^{\ast}(G)\cap B_{1}(VN(G))$. Otherwise, we consider $\mu'=\frac{\tilde{\mu}}{||T||}$ and $\lambda(\mu')$. In any case, there is a measure $\mu\in B_{1}(M(G))\cap M(E)$ such that $\lambda(\mu)\in C_{r}^{\ast}(G)\cap B_{1}(VN(G))$ and $\lambda(\mu)\neq 0$. Thus, $E\in\pi_{2}(P)$\end{proof}

\begin{thm}\label{M0} Let $G$ be a connected abelian Lie group. Then, $\mathcal{U}_{0}(G)$ is $\Pi_{1}^{1}$-complete.\end{thm}
\begin{proof} By Theorem \ref{coanM0} it is enough to prove that $\mathcal{U}_{0}(G)$ is $\Pi_{1}^{1}$-hard. \\ \\ \noindent \textbf{Step 1}: By Theorem \ref{prodM0} and considering the map $h:\mathcal{F}(\mathbb{T})\mapsto\mathcal{F}(\mathbb{T}^{k})$ given by $E\mapsto E^k$ we conclude that $\mathcal{U}_{0}(\mathbb{T}^{k})$ is $\Pi_{1}^{1}$-hard. \\ \\ \noindent \textbf{Step 2} : Let $G_1$ and $G_2$ be locally compact abelian groups, $\varphi:G_{1}\rightarrow G_2$ be an isomorphism (of topological groups) and $\mu\in M(G_1)$. This induces a pushforward measure $\varphi_{\ast}\mu$ on $G_2$ and an isomorphism $\varphi_{\ast}:\hat{G_1}\rightarrow\hat{G_2}$ given by $\chi\mapsto\chi\circ\varphi^{-1}$ such that : $$\hat{\mu}(\chi)=\int_{G_{1}}\overline{\chi(x)}d\mu(x)=\int_{G_{2}}\overline{\varphi_{\ast}(\chi)(y)}d\varphi_{\ast}\mu(y)=\widehat{\varphi_{\ast}\mu}(\varphi_{\ast}(\chi))$$ Hence, if $\hat{\mu}\in C_{0}(\hat{G_1})$ it follows that $\widehat{\varphi_{\ast}\mu}\in C_{0}(\hat{G_2})$. If $E\subseteq G_1$ is a $M_0$-set there is some $\mu\in M(E)$ such that $\mu\neq 0$ is Rajchman and thus, $\varphi_{\ast}(\mu)\in M(\varphi(E))$ and is a non zero Rajchman measure implying that $\varphi(E)$ is also a $M_{0}$-set. Analogously, we may conclude that $E\subseteq G_1$ is a $M_{0}$-set if and only if $\varphi(E)$ is a $M_0$-set. \\ \\ \noindent \textbf{Step 3} : Any such $G$ is isomorphic to $\mathbb{R}^{n}\times\mathbb{T}^{m}$ for some integers $n,m\geq 0$ and consequently, by Step 2, it suffices to prove that $\mathcal{U}_{0}(\mathbb{R}^{n}\times\mathbb{T}^{m})$ is $\Pi_{1}^{1}$-hard. In order to do so, we appeal to Corollary \ref{quocienteM0} and consider the Borel reduction $f:\mathcal{F}(\mathbb{T}^{n+m})\rightarrow\mathcal{F}(\mathbb{R}^{n}\times\mathbb{T}^{m})$ given by $E\mapsto q^{-1}(E)$, where $q:\mathbb{R}^{n}\times\mathbb{T}^{m}\rightarrow\mathbb{T}^{n+m}$ is the quotient map. \end{proof}

\section{Acknowledgements} \noindent The author would like to express gratitude to L. Turowska for her extremely helpful suggestions.

\end{document}